\documentclass[12pt]{article}
\usepackage{amsmath,amsthm, amssymb, amsfonts,accents,nccmath}
\usepackage{enumitem}
\usepackage{array}
\usepackage{lipsum}
\usepackage{mathtools}
\usepackage[toc,page]{appendix}

\usepackage[english]{babel}
\usepackage{cite}
\usepackage{dsfont}
\usepackage{booktabs}
\usepackage[linesnumbered,commentsnumbered,ruled,vlined]{algorithm2e}
\usepackage{tikz}
\usetikzlibrary{decorations.pathreplacing,calligraphy}
\usepackage[utf8]{inputenc}
\usepackage[english]{babel}
\DeclarePairedDelimiter\ceil{\lceil}{\rceil}

\usepackage{caption}
\usepackage{subcaption}

\newcommand\restr[2]{{
  \left.\kern-\nulldelimiterspace 
  #1 
  \vphantom{\big|} 
  \right|_{#2} 
  }}

\usepackage{hyperref}
\hypersetup{
    colorlinks=true,
    linkcolor=blue,
    filecolor=darkmagenta,      
    urlcolor=cyan,
    citecolor=red
} 
\makeatletter
\def\BState{\State\hskip-\ALG@thistlm}
\makeatother
\numberwithin{equation}{section}

\makeatletter
\tikzset{
	dot diameter/.store in=\dot@diameter,
	dot diameter=2pt,
	dot spacing/.store in=\dot@spacing,
	dot spacing=9pt,
	dots/.style={
		line width=\dot@diameter,
		line cap=round,
		dash pattern=on 0pt off \dot@spacing
	}
}
\makeatother

\newtheorem{theorem}{Theorem}[section]

\newtheorem{lem}[theorem]{Lemma}

\newtheorem{rem}[theorem]{Remark}

\title{The extremal Sombor index of trees with a given\\ dissociation number $\varphi$}
\author{
	Joyentanuj Das\thanks{Department of Mathematics, College of Engineering and Technology SRM Institute of Science and Technology, Kattankulathur, Chennai 603203, India.\\
	Emails: joyentanuj@gmail.com, joyentad@srmist.edu.in.}
}
\date{}

\begin{document}

\maketitle

\begin{abstract}
The Sombor index is a topological index in graph theory defined by Gutman in 2021. In this article, we find the maximum Sombor index of trees of order $\mathbf{n}$ with a given dissociation number $\varphi$, where $\ceil*{\frac{2\mathbf{n}}{3}} \leq \varphi(G) \leq \mathbf{n}-1$. We also provide the unique graph among the chosen class where the maximum Sombor index is attained.
\end{abstract}



\section{Introduction}
Let $G=(V(G),E(G))$  be a finite, simple, connected graph with $V(G)$ as the set of vertices and $E(G)$ as the set of edges in $G$. We simply write $G=(V,E)$ if there is no scope of confusion. We write $u\sim v$ to indicate that the vertices $u,v \in V$ are adjacent in $G$. The degree
of the vertex $v$, denoted by $d_G(v)$ (or simply $d(v)$), equals the number of vertices in $V$ that are adjacent to $v$.  A graph $H$ is said to be a subgraph of $G$ if $V(H) \subset V(G)$ and $E(H) \subset E(G)$. For any subset $S \subset V (G)$, a subgraph $H$ of $G$ is said to be an induced subgraph with vertex set $S$, if $H$ is a maximal subgraph of $G$ with vertex set $V(H)=S$.

For any two vertices $u,v \in V(G)$, we will use $d(u,v)$ to denote the distance between the vertices $u$ and $v$, i.e. the length of the shortest path connecting $u$ and $v$. For a vertex $u \in V(G)$, we use $N(u)$ to denote the set of neighbors of $u$, \textit{i.e.} the set of vertices that are adjacent to $u$. A vertex $u \in V(G)$ is said to be a pendant vertex if $d(u) = 1$. Since a pendant vertex $u \in V(G)$ has $d(u) = 1$ it is adjacent to exactly one vertex in $G$, say $v$ and this vertex $v$ is said to be the support vertex of $u$.

A graph $G=(V,E)$ is said to be bipartite if the vertex set $V$ can be partitioned into two subsets $M$ and $N$ such that  $E\subset M\times N$. A bipartite graph with vertex partitions $M$ and $N$  is said to be a complete bipartite graph if every vertex of $M$ is adjacent to every vertex of $N$ and moreover, if $|M| = m$  and $|N| = n$, then the  complete bipartite graph is denoted by  $K_{m,n}$. Let $S_n = K_{1,n-1}$ be the star graph of order $n$. Set $V(S_n) = \{v_0,v_1,\cdots,v_{n-1}\}$ such that $d(v_0) = n-1$ and $d(v_i) = 1$ for all $1 \le i \le n-1$. The vertex $v_0$ is called the central vertex of $S_n$.

A vertex cover of a graph is a set of vertices that includes at least one endpoint of every edge of the graph. Given a graph \( G \) and a positive integer \( k \), a subset \( S \) of the vertices of \( G \) is a vertex \( k \)-path cover if, for every path on \( k \) vertices in \( G \), it contains a vertex from \( S \). The minimum cardinality of a vertex \( k \)-path cover is denoted by \( \psi_k(G) \). The graph invariant \( \psi_k(G) \) was introduced in \cite{Bresar,Bresar1} motivated by the problem of ensuring data integrity of communications in wireless sensor networks. One can see that \( \psi_2(G) \) is the same as the minimum cardinality of a vertex cover of \( G \); hence, the concept of vertex \( k \)-path cover is a generalization of vertex cover.

A set \( I \) of vertices in a graph \( G \) is an independent set if no pair of vertices of \( I \) are adjacent. The independence number of \( G \), denoted by \( \alpha(G) \), is the maximum cardinality of an independent set in \( G \). We have a nice relation between \( \psi_2(G) \) and \( \alpha(G) \) as follows:

\[
\psi_2(G) + \alpha(G) = |V(G)|.
\]

A set \( D \) of vertices in a graph \( G \) is a dissociation set if the induced subgraph with vertex set \( D \) has maximum degree at most 1. A maximum dissociation set of \( G \) is a dissociation set with maximum cardinality. The dissociation number of \( G \), denoted by \( \varphi(G) \), is the cardinality of a maximum dissociation set. Note that \( I \subset D \). Clearly one can also observe that \( \psi_3(G) \) and \( \varphi(G) \) are related as follows:

\[
\psi_3(G) + \varphi(G) = |V(G)|.
\]

In~\cite{Gut1}, Gutman defined a new topological index under the name Sombor index. For a graph G, its Sombor index is defined as $$SO(G) = \sum_{u \sim v} \sqrt{d(u)^2 + d(v)^2}.$$ Note that, in $SO(G)$, the summation is over all edges of $G$. This new index attracted many researchers within a short period of time and some of them were able to show applications of the Sombor index, for example see~\cite{Reti}. For articles related to Sombor index the readers can refer to \cite{Ali,Cruz,Das,Das1,Das2,Fili,Gut1,Gut2,Horo,Red,Reti,Sun,Wang,Zhou,Zhou1}.

In graph theory, studying extremal graphs and indices for a class of graphs with a given parameter is a very interesting problem. For example, one may refer to \cite{Boro,Das1,Das2,Sun,Tom,Vas,Zhou,Zhou1}. For a fixed independence number, extremal graphs have been studied for indices like the Harary index in \cite{Boro}, the first Zagreb index in \cite{Vas} and functions like the Extremal vertex-degree function in \cite{Tom}.

In \cite{Zhou}, the authors have studied the extremal Sombor index of trees and unicyclic graphs with a given matching number. In \cite{Zhou1}, the authors have studied the Sombor index of trees and unicyclic graphs with a given maximum degree. In~\cite{Sun}, the authors found the maximum and the minimum Sombor index of trees with fixed domination number. In~\cite{Das3}, the author found the maximum Sombor index of trees with a given independence number.

From the definition of the dissociation number, it is clear that the dissociation number is an immediate generalization of the independence number. Several interesting results have been published related to dissociation number (for example see~\cite{Das4, Hunag, Orlovich, Tu, Tu1}). In this article, we find the maximum Sombor index of trees of order $\mathbf{n}$ with a given dissociation number $\varphi$, where $\ceil*{\frac{2\mathbf{n}}{3}} \leq \alpha \leq \mathbf{n}-1$ and also provide the unique graph where the maximum Sombor index is attained.

This article is organized as follows: In Section~\ref{sec:notations}, we state the preliminary results, state some basic inequalities and a result related to trees and its maximal independence set that will be used in the later section. In Section~\ref{sec:main}, we build up the necessary tools required to prove the main result and lastly in Theorem~\ref{thm:main}, we state and prove the main result of the article.

\section{Notation and Preliminaries}\label{sec:notations}
The next two lemmas directly follow from verifying $f'(x) >0$ and $g'(x) < 0$ and these will be used again and again in the future proofs.
\begin{lem}\label{lem:f}
	For $x \ge 1$, if $f(x) = \sqrt{(x+c)^2+d^2}-\sqrt{x^2+d^2}$, where $c,d$ are positive integers, then $f(x)$ is a monotonically increasing function of $x$.
\end{lem}

\begin{lem}\label{lem:g}
	For $x \ge 1$, if $g(x) = \sqrt{c^2+ x^2}-\sqrt{d^2+x^2}$, where $c,d$ are positive integers and $c>d$, then $g(x)$ is a monotonically decreasing function of $x$.
\end{lem}

A vertex is called a quasi-pendant vertex if it is adjacent to a pendant vertex. For a graph $G$, let $P(G)$ and $Q(G)$ denote the set of all pendant and quasi-pendant vertices, respectively. Moreover, let $Q_2(G)$ be the set of all quasi-pendant vertices of degree $2$ in $G$. In \cite{Hunag}, the authors showed that there exists a maximum dissociation set of $G$ such that it contains all vertices of $P(G) \cup Q_2(G)$. The result is stated as follows:

\begin{lem}\cite{Hunag}\label{lem:dis}
	Let $G$ be a graph with order $n \ge 5$. Then there exists a maximum dissociation set
	$D(G)$ such that $P(G) \cup Q_2(G) \subseteq D(G)$.
\end{lem}

\begin{lem}\label{lem:q2=0}
	Let $T$ be the tree of order $\mathbf{n}$ and dissociation number $\varphi$ with maximum Sombor index then, there is no quasi-pendant vertex of degree $2$, i.e. $Q_2(T) = \emptyset$.
\end{lem}

\begin{proof}
	Suppose $Q_2(T) \neq \emptyset$ and $v \in Q_2(T)$. Let $N(v) = \{u,w\}$, where $w$ is a pendant vertex and $T^*$ be the graph obtained from $T$ by deleting the edge $v \sim w$ and adding the edge $u \sim w$. Then,
	\begin{align*}
		SO(T^*) - SO(T) &> \sqrt{(d(u)+1)^2 + 1^2} - \sqrt{d(u)^2 + 2^2} + \sqrt{(d(u)+1)^2 + 1^2} - \sqrt{2^2 + 1^2} > 0,
	\end{align*}
	where, the last inequality follows from the fact that $d(u) \ge 2$. Note that, $\varphi(T^*) = \varphi(T)$, but $SO(T^*) > SO(T)$, which is a contradiction to the maximality of $T$. Hence $Q_2(T) = \emptyset$.
\end{proof}

\section{Trees with a given dissociation number}\label{sec:main}
Let $T(\mathbf{n},\varphi)$ be the collection of trees of order $\mathbf{n}$ and dissociation number $\varphi$. Observe that, if $\varphi = \mathbf{n}-1$ then $T(\mathbf{n},\mathbf{n}-1) = \{S_\mathbf{n}\}$ and $SO(S_\mathbf{n}) = (\mathbf{n}-1)\sqrt{(\mathbf{n}-1)^2 + 1} = \varphi \sqrt{\varphi^2 + 1}$, hence we consider the case when $\varphi \le \mathbf{n}-2$.

Let \(T_1(\mathbf{n},\varphi)\) be a set of trees on $\mathbf{n}$ vertices obtained from $S_{\mathbf{n}-\varphi}$ by attaching at least two pendant edges to each vertex of $S_{\mathbf{n}-\varphi}$ so that the total number of pendant vertices of the trees in $S_{\mathbf{n}-\varphi}$ is \( \varphi \). Let $T_2(\mathbf{n},\varphi)$ be a set of trees on $\mathbf{n}$ vertices obtained from $S_{\mathbf{n}-(\varphi-1)}$ by attaching at least two pendant edges to each vertices in \(\{v_1, \ldots, v_{\mathbf{n}-\varphi}\}\) and attaching one pendant edge to \( v_0 \) so that the total number of leaves of the trees in $T_2(\mathbf{n},\alpha)$ is \( \varphi - 1 \), where \( \varphi \notin \left\{ \frac{2\mathbf{n} }{3}, \frac{2\mathbf{n} + 1}{3} \right\} \). Let $T_3(\mathbf{n},\varphi)$ be a set of trees on \( n \) vertices obtained from $S_{\mathbf{n}-(\varphi-1)}$ by attaching at least two pendant edges to each vertex in \(\{v_1, \ldots, v_{\mathbf{n}-\varphi}\}\) so that the total number of leaves of the trees in $T_3(\mathbf{n},\varphi)$ is \( \varphi - 1 \), where \( \varphi \neq \frac{2\mathbf{n}}{3}\). Note that $T_1(\mathbf{n},\varphi) \cup T_2(\mathbf{n},\varphi) \cup T_3(\mathbf{n},\varphi) \subset T(\mathbf{n},\alpha)$. For reference, one can see Figure~\ref{fig:tree-1}.

\begin{figure}[ht]
	\centering
	\begin{subfigure}[t]{0.5\textwidth}
		\centering
		\begin{tikzpicture}[scale = 1.4]
			\draw[fill=red] (0,0) circle (2pt);
			\draw[fill=black] (1,1) circle (2pt);
			\draw[fill=black] (1,-1) circle (2pt);
			
			\draw[fill=black] (2,0.5) circle (2pt);
			\draw[fill=black] (2,1.5) circle (2pt);
			\draw[fill=black] (2,-0.5) circle (2pt);
			\draw[fill=black] (2,-1.5) circle (2pt);
			
			\draw[fill=black] (-1,0.5) circle (2pt);
			\draw[fill=black] (-1,-0.5) circle (2pt);
			
			\draw[thick] (0,0) -- (1,1);
			\draw[thick] (2,0.5) -- (1,1);
			\draw[thick] (2,1.5) -- (1,1);
			\draw[thick] (0,0) -- (1,-1) -- (2,-1.5);
			\draw[thick] (2,-0.5) -- (1,-1);
			
			\draw[thick] (0,0) -- (-1,0.5);
			\draw[thick] (0,0) -- (-1,-0.5);
			
			\draw[dotted] (1,0.7) -- (1,-0.7);
			\draw[dotted] (-1,0.4) -- (-1,-0.4);
			\draw[dotted] (2,0.5) -- (2,1.5);
			\draw[dotted] (2,-0.5) -- (2,-1.5);
			
			\node at (0,-0.3) {$v_0$};
			\node at (1,1+0.2) {$v_1$};
			\node at (1,-1-0.3) {$v_{\mathbf{n}-(\varphi+1)}$};		
		\end{tikzpicture}
		\caption{\(T \in T_1(\mathbf{n},\varphi)\)}
	\end{subfigure}%
	~ 
	\begin{subfigure}[t]{0.5\textwidth}
		\centering
		\begin{tikzpicture}[scale = 1.4]
			\draw[fill=red] (0,0) circle (2pt);
			\draw[fill=black] (1,1) circle (2pt);
			\draw[fill=black] (1,-1) circle (2pt);
			
			\draw[fill=black] (2,0.5) circle (2pt);
			\draw[fill=black] (2,1.5) circle (2pt);
			\draw[fill=black] (2,-0.5) circle (2pt);
			\draw[fill=black] (2,-1.5) circle (2pt);
			
			\draw[fill=black] (-1,0) circle (2pt);
			
			\draw[thick] (0,0) -- (1,1);
			\draw[thick] (2,0.5) -- (1,1);
			\draw[thick] (2,1.5) -- (1,1);
			\draw[thick] (0,0) -- (1,-1) -- (2,-1.5);
			\draw[thick] (2,-0.5) -- (1,-1);
			
			\draw[thick] (0,0) -- (-1,0);
			
			\draw[dotted] (1,0.7) -- (1,-0.7);
			\draw[dotted] (2,0.5) -- (2,1.5);
			\draw[dotted] (2,-0.5) -- (2,-1.5);
				
			\node at (0,-0.3) {$v_0$};
			\node at (1,1+0.2) {$v_1$};
			\node at (1,-1-0.3) {$v_{\mathbf{n}-\varphi}$};	
		\end{tikzpicture}
		\caption{\(T \in T_2(\mathbf{n},\varphi)\)}
	\end{subfigure}
		~ 
	\begin{subfigure}[t]{0.5\textwidth}
		\centering
		\begin{tikzpicture}[scale = 1.4]
			\draw[fill=red] (0,0) circle (2pt);
			\draw[fill=black] (1,1) circle (2pt);
			\draw[fill=black] (1,-1) circle (2pt);
			
			\draw[fill=black] (2,0.5) circle (2pt);
			\draw[fill=black] (2,1.5) circle (2pt);
			\draw[fill=black] (2,-0.5) circle (2pt);
			\draw[fill=black] (2,-1.5) circle (2pt);
			
			\draw[thick] (0,0) -- (1,1);
			\draw[thick] (2,0.5) -- (1,1);
			\draw[thick] (2,1.5) -- (1,1);
			\draw[thick] (0,0) -- (1,-1) -- (2,-1.5);
			\draw[thick] (2,-0.5) -- (1,-1);

			\draw[dotted] (1,0.7) -- (1,-0.7);
			\draw[dotted] (2,0.5) -- (2,1.5);
			\draw[dotted] (2,-0.5) -- (2,-1.5);
			
			\node at (0,-0.3) {$v_0$};
			\node at (1,1+0.2) {$v_1$};
			\node at (1,-1-0.3) {$v_{\mathbf{n}-\varphi}$};		
		\end{tikzpicture}
		\caption{\(T \in T_3(\mathbf{n},\varphi)\)}
	\end{subfigure}
	\caption{Trees in \(T_1(\mathbf{n},\varphi)\), \(T_2(\mathbf{n},\varphi)\) and \(T_3(\mathbf{n},\varphi)\) with some labeled vertices.}  \label{fig:tree-1}
\end{figure}

Let $T^*(\mathbf{n},\varphi)$ be the tree of order $\mathbf{n}$ obtained from the star $S_{\mathbf{n}-\varphi}$ by attaching exactly two pendant edge to each of the vertices $\{v_1,v_2,\cdots,v_{\mathbf{n}-(\varphi+1)}\} \subset V(S_{\mathbf{n}-\alpha})$ and attaching $3\varphi-2(\mathbf{n}-1)$ pendant vertices to the central vertex $v_0$ of $S_{\mathbf{n}-\alpha}$. Observe that, $T^*(\mathbf{n},\alpha) \in T_1(\mathbf{n},\alpha) \subset T(\mathbf{n},\alpha)$. For reference, one can see Figure~\ref{fig:tree}.

\begin{figure}[ht]
	\centering
	\begin{tikzpicture}[scale = 1.4,trim left=-3cm]
		\draw[fill=red] (0,0) circle (2pt);
		\draw[fill=black] (1,1) circle (2pt);
		\draw[fill=black] (1,-1) circle (2pt);
		
		\draw[fill=black] (2,0.5) circle (2pt);
		\draw[fill=black] (2,1.5) circle (2pt);
		\draw[fill=black] (2,-0.5) circle (2pt);
		\draw[fill=black] (2,-1.5) circle (2pt);
		
		\draw[fill=black] (-1,0.5) circle (2pt);
		\draw[fill=black] (-1,-0.5) circle (2pt);
		
		\draw[thick] (0,0) -- (1,1);
		\draw[thick] (2,0.5) -- (1,1);
		\draw[thick] (2,1.5) -- (1,1);
		\draw[thick] (0,0) -- (1,-1) -- (2,-1.5);
		\draw[thick] (2,-0.5) -- (1,-1);
		
		\draw[thick] (0,0) -- (-1,0.5);
		\draw[thick] (0,0) -- (-1,-0.5);
		
		\draw[dotted] (1,0.7) -- (1,-0.7);
		\draw[dotted] (-1,0.4) -- (-1,-0.4);

		\node at (0,-0.3) {$v_0$};
		\node at (1,1+0.2) {$v_1$};
		\node at (1,-1-0.3) {$v_{\mathbf{n}-(\varphi+1)}$};
		
		\node at (-2.7,0) {$3\varphi-2(\mathbf{n}-1)$};
		
		\draw [decorate, 
		decoration = {brace, raise=5pt,	amplitude=8pt}] (-1.2,-0.6) --  (-1.2,0.6);

	\end{tikzpicture}
\caption{$T^*(\mathbf{n},\varphi)$} \label{fig:tree}
\end{figure}

\begin{rem}
	Let $T^*(\mathbf{n},\alpha) \in T_1(\mathbf{n},\alpha) \subset T(\mathbf{n},\alpha)$ be defined as above, then
	\begin{align*}
		SO(T^*(\mathbf{n},\alpha)) &= (3\varphi-2(\mathbf{n}-1)) \sqrt{(2\varphi-\mathbf{n}+1)^2 + 1}\\
		&\ \ + (\mathbf{n}-(\varphi+1)) [\sqrt{(2\varphi-\mathbf{n}+1)^2 + 2^2} + 2\sqrt{2^2 + 1}],
	\end{align*}
	and since the vertex $v_0$ has atleast two pendant vertices, we have $3\varphi-2(\mathbf{n}-1) \ge 2$, \textit{i.e.} $\varphi \ge \ceil*{\frac{2\mathbf{n}}{3}}$.
\end{rem}

\begin{lem}\label{lem:1}
	Let $T \in T(\mathbf{n},\varphi)$ be a tree with maximal Sombor index, then $T \in T_1(\mathbf{n},\varphi) \cup T_2(\mathbf{n},\varphi) \cup T_3(\mathbf{n},\varphi)$.
\end{lem} 

\begin{proof}
Suppose $T \in T(\mathbf{n},\varphi) \setminus (T_1(\mathbf{n},\varphi) \cup T_2(\mathbf{n},\varphi)  \cup T_3(\mathbf{n},\varphi))$ and $T'$ be the tree obtained from $T$ by deleting the pendant along with the edges. Then, the number of support vertices in $T'$ is at least $2$, and we can choose two support vertices from $T'$, $u$ and $v$ such that $d(u,v)$ is the maximum in $T'$. Since $T'$ is a tree, there is a unique path between $u$ and $v$ and let $u \sim x$ and $y \sim v$ on this path.

Let $N_T(u) = \{u_1,\cdots,u_m\} \cup \{u_{m+1},\cdots,u_{m+a}\} \cup \{x\}$, where $m \ge 1$, $d(u_i) \ge 2$ for all $1 \le i \le m$ and $d(u_{m+i}) = 1$ for all $1 \le i \le a$. Let $N_T(v) = \{v_1,\cdots,v_n\} \cup \{v_{n+1},\cdots,v_{n+b}\} \cup \{y\}$,where $n \ge 1$, $d(v_j) \ge 2$ for all $1 \le j \le n$ and $d(v_{n+j}) = 1$ for all $1 \le j \le b$. If we delete the edges $u \sim x$ and $v \sim y$, then $T$ becomes disconnected into multiple components. Let $H$ be the component of the tree $T$ containing $x$ and $y$. Note that $d(u) = m+a+1$ and $d(v) = n+b+1$.	We prove the result by dividing it into the following cases:
\\

\underline{\textbf{Case 1:}} $a>1$ and $b>1$.
Using Lemma~\ref{lem:dis} we can choose a maximal dissociation set $D(T)$ such that it contains all the pendant and degree $2$ quasi-pendant vertices. Since $a,b > 1$, $\{u,v\} \cup \{u_1,\cdots,u_m\} \cup \{v_1,\cdots,v_n\}  \notin D(T)$. Thus we have $$\varphi(T) = \sum_{i=1}^{m}(d(u_i)-1) + \sum_{i=1}^{n}(d(v_i)-1)+a+b+\alpha(H).$$ Without loss of generality, we can assume that $d(u) \ge d(v)$.

\underline{\textbf{Subcase 1.1:}} $d(y) > d(x)$. Let $T^*$ be the graph obtained from $T$ by applying the following transformations:
\begin{itemize}
	\item Delete the edges $v \sim v_i$ for all $1 \le i \le n+(b-2)$.
	\item Add the edges $u \sim v_i$ for all $1 \le i \le n+(b-2)$.
\end{itemize}
Now, observe that after the transformation we have $\varphi(T) = \varphi(T^*)$ and 
\begin{align*}
	SO(T^*) - SO(T) &= \sqrt{d_{T^*}(u)^2 + d_{T^*}(x)^2} - \sqrt{d_{T}(u)^2 + d_{T}(x)^2} \\
	&\ \ +\sqrt{d_{T^*}(v)^2 + d_{T^*}(y)^2} - \sqrt{d_{T}(v)^2 + d_{T}(y)^2}\\
	&\ \ +\sum_{i=1}^{m+a} \sqrt{d_{T^*}(u)^2 + d_{T^*}(u_i)^2} - \sqrt{d_{T}(u)^2 + d_{T}(u_i)^2} \\
	&\ \ +\sum_{j=1}^{n+b-2} \sqrt{d_{T^*}(u)^2 + d_{T^*}(v_j)^2} - \sqrt{d_{T}(v)^2 + d_{T}(v_j)^2}\\
	&\ \ +\sqrt{d_{T^*}(v)^2 + d_{T^*}(v_{n+b-1})^2} - \sqrt{d_{T}(v)^2 + d_{T}(v_{n+b-1})^2}\\
	&\ \ +\sqrt{d_{T^*}(v)^2 + d_{T^*}(v_{n+b})^2} - \sqrt{d_{T}(v)^2 + d_{T}(v_{n+b})^2}.
\end{align*}
Note that, after the transformation only the degree of vertices $u$ and $v$ are changed, where $d_{T^*}(u) = m+n+a+b-1$ and $d_{T^*}(v) = 3$. Thus, substituting the values of degree we have	
\begin{align*}
	SO(T^*) - SO(T) &= \sqrt{(m+n+a+b-1)^2 + d(x)^2} - \sqrt{(m+a+1)^2 + d(x)^2} \\
	&\ \ +\sqrt{3^2 + d(y)^2} - \sqrt{(n+b+1)^2 + d(y)^2}\\
	&\ \ +\sum_{i=1}^{m+a} \sqrt{(m+n+a+b-1)^2 + d(u_i)^2} - \sqrt{(m+a+1)^2 + d(u_i)^2} \\
	&\ \ +\sum_{j=1}^{n+b-2} \sqrt{(m+n+a+b-1)^2 + d(v_j)^2} - \sqrt{(n+b+1)^2 + d(v_j)^2}\\
	&\ \ + 2(\sqrt{3^2 + 1^2} - \sqrt{(n+b+1)^2 + 1^2}).
\end{align*}
Since $a > 1$, we have at least two pendant vertex attached to $u$. Also, since the function $h(x) = \sqrt{x^2+a^2}$ is monotonically increasing for $x>0$ we have $$\sqrt{(m+n+a+b-1)^2 + d(u_i)^2} - \sqrt{(m+a+1)^2 + d(u_i)^2} > 0,$$ for $1 \le i \le m+a$ and $$\sqrt{(m+n+a+b-1)^2 + d(v_j)^2} - \sqrt{(n+b+1)^2 + d(v_j)^2} > 0,$$ for $1 \le j \le n+b-1$. Thus, combining we have
\begin{align*}
	SO(T^*) - SO(T) &> \sqrt{(m+n+a+b-1)^2 + d(x)^2} - \sqrt{(m+a+1)^2 + d(x)^2} \\
	&\ \ +\sqrt{3^2 + d(y)^2} - \sqrt{(n+b+1)^2 + d(y)^2}\\
	&\ \ +2(\sqrt{(m+n+a+b-1)^2 +1} - \sqrt{(m+a+1)^2 + 1})\\
	&\ \ +2(\sqrt{3^2 +1} - \sqrt{(n+b+1)^2 + 1}).
\end{align*}
Using Lemma~\ref{lem:f} and substituting $c = n+b-2$ and $d = 1$, we have $$\sqrt{(m+n+a+b-1)^2 +1} - \sqrt{(m+a+1)^2 + 1} > \sqrt{(n+b+1)^2 + 1} - \sqrt{3^2 +1},$$ since $f(m+a+1) > f(3)$.  Next, again using Lemma~\ref{lem:f} and substituting $c = n+b-2$ and $d = d(x)$, we have $$\sqrt{(m+n+a+b-1)^2 + d(x)^2} - \sqrt{(m+a+1)^2+d(x)^2} > \sqrt{(n+b+1)^2 + d(x)^2} - \sqrt{3^2+d(x)^2},$$ since $f(m+a+1) > f(3)$. Finally, using Lemma~\ref{lem:g} and $c = n+b+1$ and $d = 3$, we have $$ \sqrt{(n+b+1)^2 + d(x)^2} - \sqrt{2^2+d(x)^2} \ge \sqrt{(n+b+1)^2 + d(y)^2} - \sqrt{2^2 + d(y)^2},$$ since $d(y) > d(x)$ and hence $g(d(x)) > g(d(y))$. Thus, combining the above inequalities we have $SO(T^*) - SO(T) > 0$, which is a contradiction.

\underline{\textbf{Subcase 1.2:}} $d(x) \ge d(y)$. Let $T^*$ be the graph obtained from $T$ by applying the following transformations:
\begin{itemize}
	\item Delete the edges $u \sim x$ and $v \sim y$.
	\item Add the edges $u \sim y$ and $v \sim x$.
	\item Delete the edges $v \sim v_i$ for all $1 \le i \le n+(b-2)$.
	\item Add the edges $u \sim v_i$ for all $1 \le i \le n+(b-2)$.
\end{itemize}
Note that, after the transformation we have $\varphi(T) = \varphi(T^*)$ and
\begin{align*}
	SO(T^*) - SO(T) &= \sqrt{d_{T^*}(u)^2 + d_{T^*}(y)^2} - \sqrt{d_{T}(v)^2 + d_{T}(x)^2} \\
	&\ \ +\sqrt{d_{T^*}(v)^2 + d_{T^*}(x)^2} - \sqrt{d_{T}(u)^2 + d_{T}(x)^2}\\
	&\ \ +\sum_{i=1}^{m+a} \sqrt{d_{T^*}(u)^2 + d_{T^*}(u_i)^2} - \sqrt{d_{T}(u)^2 + d_{T}(u_i)^2} \\
	&\ \ +\sum_{j=1}^{n+b-2} \sqrt{d_{T^*}(u)^2 + d_{T^*}(v_j)^2} - \sqrt{d_{T}(v)^2 + d_{T}(v_j)^2}\\
	&\ \ +\sqrt{d_{T^*}(v)^2 + d_{T^*}(v_{n+b-1})^2} - \sqrt{d_{T}(v)^2 + d_{T}(v_{n+b-1})^2}\\
	&\ \ +\sqrt{d_{T^*}(v)^2 + d_{T^*}(v_{n+b})^2} - \sqrt{d_{T}(v)^2 + d_{T}(v_{n+b})^2}.
\end{align*}
Note that, after the transformation only the degree of vertices $u$ and $v$ are changed, where $d_{T^*}(u) = m+n+a+b-1$ and $d_{T^*}(v) = 3$. Thus, substituting the values of degree in the above expression we have
\begin{align*}
	SO(T^*) - SO(T) &= \sqrt{(m+n+a+b-1)^2 + d(y)^2} - \sqrt{(n+b+1)^2 + d(y)^2} \\
	&\ \ +\sqrt{3^2 + d(x)^2} - \sqrt{(m+a+1)^2 + d(x)^2}\\
	&\ \ +\sum_{i=1}^{m+a} \sqrt{(m+n+a+b-1)^2 + d(u_i)^2} - \sqrt{(m+a+1)^2 + d(u_i)^2} \\
	&\ \ +\sum_{j=1}^{n+b-2} \sqrt{(m+n+a+b-1)^2 + d(v_j)^2} - \sqrt{(n+b+1)^2 + d(v_j)^2}\\
	&\ \ +2(\sqrt{3^2 + 1^2} - \sqrt{(n+b+1)^2 + 1^2}).
\end{align*}
Since $a > 0$, we have at least two pendant vertex attached to $u$. Also, the function $h(x) = \sqrt{x^2+a^2}$ is monotonically increasing for $x>0$. Thus, we have
\begin{align*}
	SO(T^*) - SO(T) &>\sqrt{(m+n+a+b-1)^2 + d(y)^2} - \sqrt{(n+b+1)^2 + d(y)^2}\\
	&\ \ +\sqrt{3^2 + d(x)^2} - \sqrt{(m+a+1)^2 + d(x)^2}\\
	&\ \ +2(\sqrt{(m+n+a+b-1)^2 +1} - \sqrt{(m+a+1)^2 + 1})\\
	&\ \ +2(\sqrt{3^2 +1} - \sqrt{(n+b+1)^2 + 1}).
\end{align*}
From the previous subcase, we have $$\sqrt{(m+n+a+b-1)^2 +1} - \sqrt{(m+a+1)^2 + 1} > \sqrt{(n+b+1)^2 + 1} - \sqrt{3^2 +1}.$$ Now, using Lemma~\ref{lem:f} and substituting $c = m+a-1$ and $d = d(y)$, we have $$\sqrt{(m+n+a+b-1)^2 + d(y)^2} - \sqrt{(n+b+1)^2 + d(y)^2} > \sqrt{(m+a+1)^2 + d(y)^2} - \sqrt{3^2 + d(y)^2},$$ since $f(n+b+1) > f(3)$. Finally, using Lemma~\ref{lem:g} and $c = m+a+1$ and $d = d(y)$, we have $$\sqrt{(m+a+1)^2 + d(y)^2} - \sqrt{3^2 + d(y)^2} \ge \sqrt{(m+a+1)^2+d(x)^2} - \sqrt{3^2 + d(x)^2},$$ since $d(x) \ge d(y)$ and hence $g(d(y)) \ge g(d(x))$. Thus, combining the above inequalities we have $SO(T^*) - SO(T) > 0$, which is a contradiction.
\\

\underline{\textbf{Case 2:}} $a=1$ but $b>1$ or $a >1$ but $b=1$.

Without loss of generality, we can assume that $a = 1$ but $b>1$. Using Lemma~\ref{lem:dis} we can choose a maximal dissociation set $D(T)$ such that it contains all the pendant and the degree $2$ quasi-pendant vertices. Since $a= 1$ but $b > 1$, $\{v\} \cup \{u_1,\cdots,u_m\} \cup \{v_1,\cdots,v_n\}  \notin D(T)$. Thus, we have $$\varphi(T) = \sum_{i=1}^{m}(d(u_i)-1) + \sum_{i=1}^{n}(d(v_i)-1)+b+\varphi(H\cup \{u\}).$$ Let $T^*$ be the graph obtained from $T$ by applying the following transformations:
\begin{itemize}
	\item Delete the edges $u \sim u_i$ for all $1 \le i \le m$.
	\item Add the edges $v \sim u_i$ for all $1 \le i \le m$.
\end{itemize}
Note that, after the transformation we have $\varphi(T) = \varphi(T^*)$ and we have
\begin{align*}
	SO(T^*) - SO(T) &= \sqrt{d_{T^*}(u)^2 + d_{T^*}(x)^2} - \sqrt{d_{T}(u)^2 + d_{T}(x)^2} \\
	&\ \ +\sqrt{d_{T^*}(v)^2 + d_{T^*}(y)^2} - \sqrt{d_{T}(v)^2 + d_{T}(y)^2}\\
	&\ \ +\sum_{i=1}^{m} \sqrt{d_{T^*}(v)^2 + d_{T^*}(u_i)^2} - \sqrt{d_{T}(u)^2 + d_{T}(u_i)^2} \\
	&\ \ +\sqrt{d_{T^*}(u)^2 + d_{T^*}(u_{m+1})^2} - \sqrt{d_{T}(u)^2 + d_{T}(u_{m+1})^2} \\
	&\ \ +\sum_{j=1}^{n+b} \sqrt{d_{T^*}(u)^2 + d_{T^*}(v_j)^2} - \sqrt{d_{T}(v)^2 + d_{T}(v_j)^2}.
\end{align*}
Note that, after the transformation only the degree of vertices $u$ and $v$ are changed, where $d_{T^*}(u) = 2$ and $d_{T^*}(v) = m+n+b+1$. Thus, substituting the values of degree in the above expression we have
\begin{align*}
	SO(T^*) - SO(T) &= \sqrt{2^2 + d(x)^2} - \sqrt{(m+2)^2 + d(x)^2} \\
	&\ \ +\sqrt{(m+n+b+1)^2 + d(y)^2} - \sqrt{(n+b+1)^2 + d(y)^2}\\
	&\ \ +\sum_{i=1}^{m} \sqrt{(m+n+b+1)^2 + d(u_i)^2} - \sqrt{(m+2)^2 + d(u_i)^2} \\
	&\ \ +\sqrt{2^2 + 1^2} - \sqrt{(m+2)^2 + 1^2} \\
	&\ \ +\sum_{j=1}^{n+b} \sqrt{(m+n+b+1)^2 + d(v_j)^2} - \sqrt{(n+b+1)^2 + d(v_j)^2}.
\end{align*}
Since $b > 1$, we have at least two pendant vertices attached to $v$. Also, the function $h(x) = \sqrt{x^2+a^2}$ is monotonically increasing for $x>0$. Thus, we have	
\begin{align*}
	SO(T^*) - SO(T) &> \sqrt{2^2 + d(x)^2} - \sqrt{(m+2)^2 + d(x)^2}\\
	&\ \ +\sqrt{2^2 + 1^2} - \sqrt{(m+2)^2 + 1^2} \\
	&\ \ +2(\sqrt{(m+n+b+1)^2 + 1} - \sqrt{(n+b+1)^2+1}).
\end{align*}
Using Lemma~\ref{lem:f} and substituting $c = m$ and $d = 1$, we have $$\sqrt{(m+n+b+1)^2 + 1} - \sqrt{(n+b+1)^2+1} > \sqrt{(m+2)^2 + 1} - \sqrt{2^2 + 1},$$ since $f(n+b+1) > f(2)$. Now using Lemma~\ref{lem:g} and substituting $c = m+2$ and $d = 2$, we have $$\sqrt{(m+2)^2 + 1} - \sqrt{2^2 + 1} > \sqrt{(m+2)^2 + d(x)^2} - \sqrt{2^2 + d(x)^2},$$ since $d(x) > 1$ and hence $g(1) > g(d(x))$. Thus, combining the above inequalities we have $SO(T^*) - SO(T) > 0$, which is a contradiction.
\\

\underline{\textbf{Case 3:}} $a=1$ and $b=1$.

Without loss of generality we assume that $d(x) \ge d(y)$. Using Lemma~\ref{lem:dis} we can choose a maximal dissociation set $D(T)$ such that it contains all the pendant and the degree $2$ quasi-pendant vertices. Since $a= 1$ and $b = 1$, $\{u_1,\cdots,u_m\} \cup \{v_1,\cdots,v_n\}  \notin D(T)$. Thus we have $$\varphi(T) = \sum_{i=1}^{m}(d(u_i)-1) + \sum_{i=1}^{n}(d(v_i)-1)+\varphi(H\cup \{u,v\}).$$ Let $T^*$ be the graph obtained from $T$ by applying the following transformations:
\begin{itemize}
	\item Delete the edges $u \sim u_i$ for all $1 \le i \le m$.
	\item Add the edges $v \sim u_i$ for all $1 \le i \le m$.
\end{itemize}
Note that after the transformation we have $\varphi(T) = \varphi(T^*)$ and we have
\begin{align*}
	SO(T^*) - SO(T) &= \sqrt{d_{T^*}(u)^2 + d_{T^*}(x)^2} - \sqrt{d_{T}(u)^2 + d_{T}(x)^2} \\
	&\ \ +\sqrt{d_{T^*}(v)^2 + d_{T^*}(y)^2} - \sqrt{d_{T}(v)^2 + d_{T}(y)^2}\\
	&\ \ +\sum_{i=1}^{m} \sqrt{d_{T^*}(v)^2 + d_{T^*}(u_i)^2} - \sqrt{d_{T}(u)^2 + d_{T}(u_i)^2} \\
	&\ \ + \sqrt{d_{T^*}(u)^2 + d_{T^*}(u_{m+1})^2} - \sqrt{d_{T}(u)^2 + d_{T}(u_{m+1})^2} \\
	&\ \ +\sum_{j=1}^{n+1} \sqrt{d_{T^*}(u)^2 + d_{T^*}(v_j)^2} - \sqrt{d_{T}(v)^2 + d_{T}(v_j)^2}.
\end{align*}
Note that, after the transformation only the degree of vertices $u$ and $v$ is changed, where $d_{T^*}(u) = 2$ and $d_{T^*}(v) = m+n+2$. Thus, substituting the values of degree in the above expression we have
\begin{align*}
	SO(T^*) - SO(T) &= \sqrt{2^2 + d(x)^2} - \sqrt{(m+2)^2 + d(x)^2} \\
	&\ \ +\sqrt{(m+n+2)^2 + d(y)^2} - \sqrt{(n+2)^2 + d(y)^2}\\
	&\ \ +\sum_{i=1}^{m} \sqrt{(m+n+2)^2 + d(u_i)^2} - \sqrt{(m+2)^2 + d(u_i)^2} \\
	&\ \ + \sqrt{2^2 + d(u_{m+1})^2} - \sqrt{(m+2)^2 + d(u_{m+1})^2} \\
	&\ \ +\sum_{j=1}^{n+1} \sqrt{(m+n+2)^2 + d(v_j)^2} - \sqrt{(n+2)^2 + d(v_j)^2}.
\end{align*}
Since the function $h(x) = \sqrt{x^2+a^2}$ is monotonically increasing for $x>0$, we have 
\begin{align*}
	SO(T^*) - SO(T) &> \sqrt{2^2 + d(x)^2} - \sqrt{(m+2)^2 + d(x)^2}\\
	&\ \ +\sqrt{(m+n+2)^2 + d(y)^2} - \sqrt{(n+2)^2 + d(y)^2}\\
	&\ \ + \sqrt{2^2 + 1^2} - \sqrt{(m+2)^2 + 1^2} \\
	&\ \ + \sqrt{(m+n+2)^2 + 1^2} - \sqrt{(n+2)^2 + 1^2}.
\end{align*}
Using Lemma~\ref{lem:f} and substituting $c = m$ and $d = d(y)$, we have $$\sqrt{(m+n+2)^2 + d(y)^2} - \sqrt{(n+2)^2+d(y)^2} > \sqrt{(m+2)^2 + d(y)^2} - \sqrt{2^2 + d(y)^2},$$ since $f(n+2) > f(2)$. Now using Lemma~\ref{lem:g} and substituting $c = m+2$ and $d = 2$, we have $$\sqrt{(m+2)^2 + d(y)^2} - \sqrt{2^2 + d(y)^2} \ge \sqrt{(m+2)^2 + d(x)^2} - \sqrt{2^2 + d(x)^2},$$ since $d(x) \ge d(y)$ and hence $g(d(y)) \ge g(d(x))$. Again, using Lemma~\ref{lem:f} and substituting $c = m$ and $d = 1$, we have $$\sqrt{(m+n+2)^2 + 1^2} - \sqrt{(n+2)^2+1^2} > \sqrt{(m+2)^2 + 1^2} - \sqrt{2^2 + 1^2},$$ since $f(n+2) > f(2)$. Thus, combining the above inequalities we have $SO(T^*) - SO(T) > 0$, which is a contradiction.

\underline{\textbf{Case 4:}} $a=0$ but $b\ge1$ or $a\ge1$ but $b=0$.

Without loss of generality, we can assume that $a = 0$ but $b\ge1$. Using Lemma~\ref{lem:dis} we can choose a maximal dissociation set $D(T)$ such that it contains all the pendant and the degree $2$ quasi-pendant vertices. On one hand, if $a= 0$ and $b > 1$, $\{v\} \cup \{u_1,\cdots,u_m\} \cup \{v_1,\cdots,v_n\}  \notin D(T)$ and on the other hand if $a= 0$ and $b = 1$, $\{u_1,\cdots,u_m\} \cup \{v_1,\cdots,v_n\}  \notin D(T)$. Thus, if $a= 0$ and $b > 1$, we have $$\varphi(T) = \sum_{i=1}^{m}(d(u_i)-1) + \sum_{i=1}^{n}(d(v_i)-1)+b+\varphi(H\cup \{u\})$$ and if $a= 0$ and $b = 1$, we have $$\varphi(T) = \sum_{i=1}^{m}(d(u_i)-1) + \sum_{i=1}^{n}(d(v_i)-1)+1+\varphi(H\cup \{u,v\}).$$  Let $T^*$ be the graph obtained from $T$ by applying the following transformations:
\begin{itemize}
	\item Delete the edges $u \sim u_i$ for all $1 \le i \le m$.
	\item Add the edges $v \sim u_i$ for all $1 \le i \le m$.
\end{itemize}
Note that, after the transformation for the both cases, where $b>1$ or $b=1$, we have $\varphi(T) = \varphi(T^*)$ and we have
\begin{align*}
	SO(T^*) - SO(T) &= \sqrt{d_{T^*}(u)^2 + d_{T^*}(x)^2} - \sqrt{d_{T}(u)^2 + d_{T}(x)^2} \\
	&\ \ +\sqrt{d_{T^*}(v)^2 + d_{T^*}(y)^2} - \sqrt{d_{T}(v)^2 + d_{T}(y)^2}\\
	&\ \ +\sum_{i=1}^{m} \sqrt{d_{T^*}(v)^2 + d_{T^*}(u_i)^2} - \sqrt{d_{T}(u)^2 + d_{T}(u_i)^2} \\
	&\ \ +\sum_{j=1}^{n+b} \sqrt{d_{T^*}(u)^2 + d_{T^*}(v_j)^2} - \sqrt{d_{T}(v)^2 + d_{T}(v_j)^2}.
\end{align*}
Note that, after the transformation only the degree of vertices $u$ and $v$ are changed, where $d_{T^*}(u) = 1$ and $d_{T^*}(v) = m+n+b+1$. Thus, substituting the values of degree in the above expression we have
\begin{align*}
	SO(T^*) - SO(T) &= \sqrt{1^2 + d(x)^2} - \sqrt{(m+1)^2 + d(x)^2} \\
	&\ \ +\sqrt{(m+n+b+1)^2 + d(y)^2} - \sqrt{(n+b+1)^2 + d(y)^2}\\
	&\ \ +\sum_{i=1}^{m} \sqrt{(m+n+b+1)^2 + d(u_i)^2} - \sqrt{(m+1)^2 + d(u_i)^2} \\
	&\ \ +\sum_{j=1}^{n+b} \sqrt{(m+n+b+1)^2 + d(v_j)^2} - \sqrt{(n+b+1)^2 + d(v_j)^2}.
\end{align*}
Since $b \ge 1$, we have at least one pendant vertex attached to $v$. Also, the function $h(x) = \sqrt{x^2+a^2}$ is monotonically increasing for $x>0$. Thus, we have	
\begin{align*}
	SO(T^*) - SO(T) &> \sqrt{1^2 + d(x)^2} - \sqrt{(m+1)^2 + d(x)^2}\\
	&\ \ +\sqrt{(m+n+b+1)^2 + 1} - \sqrt{(n+b+1)^2+1}.
\end{align*}
Using Lemma~\ref{lem:f} and substituting $c = m$ and $d = 1$, we have $$\sqrt{(m+n+b+1)^2 + 1} - \sqrt{(n+b+1)^2+1} > \sqrt{(m+1)^2 + 1} - \sqrt{1^2 + 1},$$ since $f(n+b+1) > f(1)$. Now using Lemma~\ref{lem:g} and substituting $c = m+1$ and $d = 1$, we have $$\sqrt{(m+1)^2 + 1} - \sqrt{1^2 + 1} > \sqrt{(m+1)^2 + d(x)^2} - \sqrt{1^2 + d(x)^2},$$ since $d(x) > 1$ and hence $g(1) > g(d(x))$. Thus, combining the above inequalities we have $SO(T^*) - SO(T) > 0$, which is a contradiction.
\\

\underline{\textbf{Case 5:}} $a=0$ and $b = 0$.

Without loss of generality we assume that $d(x) \ge d(y)$. Using Lemma~\ref{lem:dis} we can choose a maximal dissociation set $D(T)$ such that it contains all the pendant and degree $2$ quasi-pendant vertices. Since $a= 0$ and $b = 0$, $\{u_1,\cdots,u_m\} \cup \{v_1,\cdots,v_n\}  \notin D(T)$. Thus we have $$\varphi(T) = \sum_{i=1}^{m}(d(u_i)-1) + \sum_{i=1}^{n}(d(v_i)-1)+\varphi(H\cup \{u,v\}).$$ Let $T^*$ be the graph obtained from $T$ by applying the following transformations:
\begin{itemize}
	\item Delete the edges $u \sim u_i$ for all $1 \le i \le m$.
	\item Add the edges $v \sim u_i$ for all $1 \le i \le m$.
\end{itemize}
Note that after the transformation we have $\varphi(T) = \varphi(T^*)$ and we have
\begin{align*}
	SO(T^*) - SO(T) &= \sqrt{d_{T^*}(u)^2 + d_{T^*}(x)^2} - \sqrt{d_{T}(u)^2 + d_{T}(x)^2} \\
	&\ \ +\sqrt{d_{T^*}(v)^2 + d_{T^*}(y)^2} - \sqrt{d_{T}(v)^2 + d_{T}(y)^2}\\
	&\ \ +\sum_{i=1}^{m} \sqrt{d_{T^*}(v)^2 + d_{T^*}(u_i)^2} - \sqrt{d_{T}(u)^2 + d_{T}(u_i)^2} \\
	&\ \ +\sum_{j=1}^{n+b} \sqrt{d_{T^*}(u)^2 + d_{T^*}(v_j)^2} - \sqrt{d_{T}(v)^2 + d_{T}(v_j)^2}.
\end{align*}
Note that, after the transformation only the degree of vertices $u$ and $v$ is changed, where $d_{T^*}(u) = 1$ and $d_{T^*}(v) = m+n+1$. Thus, substituting the values of degree in the above expression we have
\begin{align*}
	SO(T^*) - SO(T) &= \sqrt{1^2 + d(x)^2} - \sqrt{(m+1)^2 + d(x)^2} \\
	&\ \ +\sqrt{(m+n+1)^2 + d(y)^2} - \sqrt{(n+1)^2 + d(y)^2}\\
	&\ \ +\sum_{i=1}^{m} \sqrt{(m+n+1)^2 + d(u_i)^2} - \sqrt{(m+1)^2 + d(u_i)^2} \\
	&\ \ +\sum_{j=1}^{n+b} \sqrt{(m+n+1)^2 + d(v_j)^2} - \sqrt{(n+1)^2 + d(v_j)^2}.
\end{align*}
Since the function $h(x) = \sqrt{x^2+a^2}$ is monotonically increasing for $x>0$, we have 
\begin{align*}
	SO(T^*) - SO(T) &> \sqrt{1^2 + d(x)^2} - \sqrt{(m+1)^2 + d(x)^2}\\
	&\ \ +\sqrt{(m+n+1)^2 + d(y)^2} - \sqrt{(n+1)^2 + d(y)^2}.
\end{align*}
Using Lemma~\ref{lem:f} and substituting $c = m$ and $d = d(y)$, we have $$\sqrt{(m+n+1)^2 + d(y)^2} - \sqrt{(n+1)^2+d(y)^2} > \sqrt{(m+1)^2 + d(y)^2} - \sqrt{1^2 + d(y)^2},$$ since $f(n+1) > f(1)$. Now using Lemma~\ref{lem:g} and substituting $c = m+1$ and $d = 1$, we have $$\sqrt{(m+1)^2 + d(y)^2} - \sqrt{1 + d(y)^2} \ge \sqrt{(m+1)^2 + d(x)^2} - \sqrt{1 + d(x)^2},$$ since $d(x) \ge d(y)$ and hence $g(d(y)) \ge g(d(x))$. Thus, combining the above inequalities we have $SO(T^*) - SO(T) > 0$, which is a contradiction.

Thus, combining all the above cases, we have that if $T \in T(\mathbf{n},\varphi)$ is a tree with maximal Sombor index, then $T \in T_1(\mathbf{n},\varphi) \cup T_2(\mathbf{n},\varphi) \cup T_3(\mathbf{n},\varphi)$.
\end{proof}

\begin{rem}
	Observe that, in the proof of Lemma~\ref{lem:1} it is possible that the vertices $x$ and $y$ are the same vertex or the vertices $u$ and $v$ are adjacent. But, in both cases, a similar argument gives the required result.
\end{rem}

\begin{lem}\label{lem:2}
	Suppose $T \in T(\mathbf{n},\varphi)$ is a tree with maximal Sombor index, then $T \in T_1(\mathbf{n},\varphi)$.
\end{lem}

\begin{proof}
	Suppose $T \in T(\mathbf{n},\varphi)$ is a tree with maximal Sombor index and $T \in T_2(\mathbf{n},\varphi)$. From the definition of the class $T_2(\mathbf{n},\varphi)$ we have $v_0$ is the central vertex of the induced star graph and $N(v_0) = \{u\} \cup  \{v_1,v_2,\cdots,v_{\mathbf{n}-\varphi}\}$. Let $N(v_1) = \{v_0\} \cup  \{w_1,w_2,\cdots,w_k\}$. Using Lemma~\ref{lem:dis}, we can choose a maximal dissociation set $D(T)$ such that it contains all the pendant and degree $2$ quasi-pendant vertices. Note that $\{v_1,v_2,\cdots,v_{\mathbf{n}-\varphi}\} \notin D(T)$. Let $T^*$ be the graph obtained from $T$ by applying the following transformations:
	\begin{itemize}
		\item Delete the edges $v_1 \sim w_i$ for all $1 \le i \le k$.
		\item Add the edges $v_0 \sim w_i$ for all $1 \le i \le k$.
	\end{itemize}
	Note that after the transformation, we have $\varphi(T) = \varphi(T^*)$ and we have 
	\begin{align*}
		SO(T^*) - SO(T) &= \sum_{i=1}^{\mathbf{n}-\alpha} \sqrt{d_{T^*}(v_0)^2 + d_{T^*}(v_i)^2} - \sqrt{d_{T}(v_0)^2 + d_{T}(v_i)^2} \\
		&\ \ +\sum_{j=1}^{k} \sqrt{d_{T^*}(v_0)^2 + d_{T^*}(w_j)^2} - \sqrt{d_{T}(v_1)^2 + d_{T}(w_j)^2}\\
		&\ \ + \sqrt{d_{T^*}(v_0)^2 + d_{T^*}(u)^2} - \sqrt{d_{T}(v_0)^2 + d_{T}(u)^2}.
	\end{align*}
	Note that, after the transformation only the degree of vertices $v_0$ and $v_1$ are changed, where $d_{T^*}(v_0) = (\mathbf{n}-\varphi+1)+k$ and $d_{T^*}(v_1) = 1$. Thus, substituting the values of degree in the above expression we have
	\begin{align*}
		SO(T^*) - SO(T) &= \sqrt{(\mathbf{n}-\varphi+k+1)^2 + 1^2} - \sqrt{(\mathbf{n}-\varphi+1)^2 + (k+1)^2} \\
		&\ \ +\sum_{i=2}^{\mathbf{n}-\varphi} \sqrt{(\mathbf{n}-\varphi+k+1)^2 + d(v_i)^2} - \sqrt{(\mathbf{n}-\varphi+1)^2 + d(v_i)^2} \\
		&\ \ +\sum_{j=1}^{k} \sqrt{(\mathbf{n}-\varphi+k+1)^2 + d(w_j)^2} - \sqrt{(k+1)^2 + d(w_j)^2}\\
		&\ \ + \sqrt{(\mathbf{n}-\varphi+k+1)^2 + 1^2} - \sqrt{(\mathbf{n}-\varphi+1)^2 + 1^2}
	\end{align*}
	Since the function $h(x) = \sqrt{x^2+a^2}$ is monotonically increasing for $x>0$, we have	
	$$SO(T^*) - SO(T) > \sqrt{(\mathbf{n}-\varphi+k+1)^2 + 1} - \sqrt{(\mathbf{n}-\varphi+1)^2 + (k+1)^2} \ge 0,$$ where the last inequality is true since $k(\mathbf{n}-\varphi+1) \ge k$. Note that, $T^* \in T_1(\mathbf{n},\alpha)$ and $SO(T^*) > SO(T)$, which is a contradiction.
	
	Next, suppose $T \in T(\mathbf{n},\varphi)$ is a tree with maximal Sombor index and $T \in T_3(\mathbf{n},\varphi)$. From the definition of the class $T_3(\mathbf{n},\varphi)$ we have $v_0$ is the central vertex of the induced star graph and $N(v_0) =  \{v_1,v_2,\cdots,v_{\mathbf{n}-\varphi}\}$. Let $N(v_1) = \{v_0\} \cup  \{w_1,w_2,\cdots,w_k\}$. Using Lemma~\ref{lem:dis}, we can choose a maximal dissociation set $D(T)$ such that it contains all the pendant and degree $2$ quasi-pendant vertices. Note that $\{v_1,v_2,\cdots,v_{\mathbf{n}-\varphi}\} \notin D(T)$. Let $T^*$ be the graph obtained from $T$ by applying the following transformations:
	\begin{itemize}
		\item Delete the edges $v_1 \sim w_i$ for all $1 \le i \le k$.
		\item Add the edges $v_0 \sim w_i$ for all $1 \le i \le k$.
	\end{itemize}
	Note that after the transformation, we have $\varphi(T) = \varphi(T^*)$ and we have 
	\begin{align*}
		SO(T^*) - SO(T) &= \sum_{i=1}^{\mathbf{n}-\alpha} \sqrt{d_{T^*}(v_0)^2 + d_{T^*}(v_i)^2} - \sqrt{d_{T}(v_0)^2 + d_{T}(v_i)^2} \\
		&\ \ +\sum_{j=1}^{k} \sqrt{d_{T^*}(v_0)^2 + d_{T^*}(w_j)^2} - \sqrt{d_{T}(v_1)^2 + d_{T}(w_j)^2}.
	\end{align*}
	Note that, after the transformation only the degree of vertices $v_0$ and $v_1$ are changed, where $d_{T^*}(v_0) = (\mathbf{n}-\varphi)+k$ and $d_{T^*}(v_1) = 1$. Thus, substituting the values of degree in the above expression we have
	\begin{align*}
		SO(T^*) - SO(T) &= \sqrt{(\mathbf{n}-\varphi+k)^2 + 1^2} - \sqrt{(\mathbf{n}-\varphi)^2 + (k+1)^2} \\
		&\ \ +\sum_{i=2}^{\mathbf{n}-\varphi} \sqrt{(\mathbf{n}-\varphi+k)^2 + d(v_i)^2} - \sqrt{(\mathbf{n}-\varphi)^2 + d(v_i)^2} \\
		&\ \ +\sum_{j=1}^{k} \sqrt{(\mathbf{n}-\varphi+k)^2 + d(w_j)^2} - \sqrt{(k+1)^2 + d(w_j)^2}.
	\end{align*}
	Since the function $h(x) = \sqrt{x^2+a^2}$ is monotonically increasing for $x>0$, we have	
	$$SO(T^*) - SO(T) > \sqrt{(\mathbf{n}-\varphi+k)^2 + 1} - \sqrt{(\mathbf{n}-\varphi)^2 + (k+1)^2} \ge 0,$$ where the last inequality is true since $k(\mathbf{n}-\varphi) \ge k$. Note that, $T^* \in T_1(\mathbf{n},\alpha)$ and $SO(T^*) > SO(T)$, which is a contradiction.
\end{proof}

Now we state and prove the main result of the article.

\begin{theorem}\label{thm:main}
	Let $\varphi \le \mathbf{n}-2$ and $T \in T(\mathbf{n},\varphi)$, then $$SO(T) \le (3\varphi-2(\mathbf{n}-1)) \sqrt{(2\varphi-\mathbf{n}+1)^2 + 1} + (\mathbf{n}-(\varphi+1)) [\sqrt{(2\varphi-\mathbf{n}+1)^2 + 2^2} + 2\sqrt{2^2 + 1}]$$ and equality is attained if and only if $T \cong T^*(\mathbf{n},\varphi)$. If $\varphi = \mathbf{n}-1$, then $T(\mathbf{n},\mathbf{n}-1) = \{S_\mathbf{n}\}$ and $SO(S_\mathbf{n}) = (\mathbf{n}-1)\sqrt{(\mathbf{n}-1)^2 + 1} = \varphi \sqrt{\varphi^2 + 1}$.
\end{theorem}

\begin{proof}
	Let $\varphi \le \mathbf{n}-2$ and $T \in T(\mathbf{n},\varphi) \ne T^*(\mathbf{n},\varphi)$ be a tree with maximal Sombor index, then using Lemmas~\ref{lem:1} and \ref{lem:2}, we have that $T \in T_1(\mathbf{n},\varphi)$. Since $T \in T(\mathbf{n},\varphi) \ne T^*(\mathbf{n},\varphi)$, there exists a vertex $v_1 \in N(v_0)$ which is not a pendant vertex and $d(v_1) \ge 4$. Let $N(v_1) = \{v_0\} \cup \{w_1,w_2,\cdots,w_k\}$, where $d(w_i) = 1$ for $1 \le i \le k$ and hence $d(v_1) = k+1$. Let $N(v_0) = \{v_1,v_2,\cdots,v_{\mathbf{n}-(\varphi+1)}\} \cup \{u_1,u_2,\cdots,u_m\}$, where $d(u_i) = 1$ for $1 \le i \le m$ and hence $d(v_0) = \mathbf{n}-(\varphi+1)+m \ge 2$. For simplicity of further calculations we assume that $d(v_0) = l+2$. Let $T^*$ be the graph obtained from $T$ by applying the following transformations:
	\begin{itemize}
		\item Delete the edges $v_1 \sim w_i$ for all $3 \le i \le k$.
		\item Add the edges $v_0 \sim w_i$ for all $3 \le i \le k$.
	\end{itemize} 
	Note that after the transformation, we have $\varphi(T) = \varphi(T^*)$ and
	\begin{align*}
		SO(T^*) - SO(T) &= \sum_{i=1}^{\mathbf{n}-(\varphi+1)} \sqrt{d_{T^*}(v_0)^2 + d_{T^*}(v_i)^2} - \sqrt{d_{T}(v_0)^2 + d_{T}(v_i)^2} \\
		&\ \ +\sum_{j=1}^{m} \sqrt{d_{T^*}(v_0)^2 + d_{T^*}(u_j)^2} - \sqrt{d_{T}(v_0)^2 + d_{T}(u_j)^2}\\
		&\ \ + \sqrt{d_{T^*}(v_1)^2 + d_{T^*}(w_1)^2} - \sqrt{d_{T}(v_1)^2 + d_{T}(w_1)^2}\\
		&\ \ + \sqrt{d_{T^*}(v_1)^2 + d_{T^*}(w_2)^2} - \sqrt{d_{T}(v_1)^2 + d_{T}(w_2)^2}\\
		&\ \ + \sum_{p=3}^{k} \sqrt{d_{T^*}(v_0)^2 + d_{T^*}(w_p)^2} - \sqrt{d_{T}(v_1)^2 + d_{T}(w_p)^2}.
	\end{align*}
	Note that, after the transformation only the degree of vertices $v_0$ and $v_1$ are changed, where $d_{T^*}(v_0) = l+k$ and $d_{T^*}(v_1) = 3$. Thus, substituting the values of degree in the above expression we have
	\begin{align*}
		SO(T^*) - SO(T) &= \sqrt{(l+k)^2 + 3^2} - \sqrt{(l+2)^2 + (k+1)^2} \\
		&\ \ +\sum_{i=2}^{\mathbf{n}-(\varphi+1)} \sqrt{(l+k)^2 + d(v_i)^2} - \sqrt{(l+2)^2 + d(v_i)^2} \\
		&\ \ +\sum_{j=1}^{m} \sqrt{(l+k)^2 + 1^2} - \sqrt{(l+2)^2 + 1^2}\\
		&\ \ + 2(\sqrt{3^2 + 1^2} - \sqrt{(k+1)^2 + 1^2})\\
		&\ \ + \sum_{p=3}^{k} \sqrt{(l+k)^2 + 1^2} - \sqrt{(k+1)^2 + 1^2}.
	\end{align*}
	Since there is at least two pendant vertex attached to $v_0$ we have $m > 1$. Also, the function $h(x) = \sqrt{x^2+a^2}$ is monotonically increasing for $x>0$. Thus, we have	
	\begin{align*}
		SO(T^*) - SO(T) &> \sqrt{(l+k)^2 + 3^2} - \sqrt{(l+1)^2 + (k+1)^2}\\
		&\ \ +2(\sqrt{(l+k)^2 + 1} - \sqrt{(l+2)^2 + 1})\\
		&\ \ +2(\sqrt{3^2 + 1} - \sqrt{(k+1)^2 + 1}).
	\end{align*}
	Using Lemma~\ref{lem:f} and substituting $c = k-2$ and $d = 1$, we have $$\sqrt{(l+k)^2 + 1} - \sqrt{(l+2)^2 + 1} \ge \sqrt{(k+1)^2 + 1} - \sqrt{3^2 + 1},$$ since $f(l+2) \ge f(3)$. Also, it follows from the inequality $lk + 1 \ge l+k$ that $$\sqrt{(l+k)^2 + 3^2} \ge \sqrt{(l+2)^2 + (k+1)^2}$$ for all $l,k \ge 1$. Hence, combining the inequalities we have $SO(T^*) > SO(T)$, which is a contradiction. Thus, $T \cong T^*(\mathbf{n},\varphi)$ and the result follows.
	
	If $\varphi = \mathbf{n}-1$, then we have already shown in the beginning of the section that $T(\mathbf{n},\mathbf{n}-1) = \{S_\mathbf{n}\}$ and $SO(S_\mathbf{n}) = (\mathbf{n}-1)\sqrt{(\mathbf{n}-1)^2 + 1} = \varphi \sqrt{\varphi^2 + 1}$.
\end{proof}

\small{

}

\end{document}